\documentclass{article}

\usepackage{amsmath,amsfonts,latexsym,amssymb,amsthm,url, titling}

\newtheorem{theorem}{Theorem}[section]
\newtheorem{proposition}[theorem]{Proposition}
\newtheorem{corollary}[theorem]{Corollary}
\newtheorem{lemma}[theorem]{Lemma}
\newtheorem{remark}[theorem]{Remark}

\DeclareMathOperator{\height}{h}
\DeclareMathOperator{\Height}{H}

\DeclareMathOperator{\im}{Im}

\DeclareMathOperator{\genus}{\mathbf{g}}
\DeclareMathOperator{\Gal}{Gal}

\newcommand{\C}{\mathbb{C}}
\newcommand{\Q}{\mathbb{Q}}
\newcommand{\R}{\mathbb{R}}
\newcommand{\Z}{\mathbb{Z}}

\newcommand{\AAA}{\mathcal{A}}

\newcommand{\LL}{\mathcal{L}}

\newcommand{\OO}{\mathcal{O}}
\newcommand{\SSS}{\mathcal{S}}

\newcommand{\eps}{\varepsilon}

\newcommand{\tilP}{\widetilde{P}}
\newcommand{\tily}{\tilde{y}}

\title{Values of algebraic functions at Liouville numbers}

\author{Yuri Bilu, Diego Marques\thanks{Supported by the Brazilian-French Network in Mathematics.}}

\date{\today}

\numberwithin{equation}{section}

\makeatletter

\renewcommand*\l@section[2]{
	\ifnum \c@tocdepth >\z@
	\addpenalty\@secpenalty
	\addvspace{0.2em \@plus\p@}
	\setlength\@tempdima{1.5em}
	\begingroup
	\parindent \z@ \rightskip \@pnumwidth
	\parfillskip -\@pnumwidth
	\leavevmode \bfseries
	\advance\leftskip\@tempdima
	\hskip -\leftskip
	#1\nobreak\hfil \nobreak\hb@xt@\@pnumwidth{\hss #2}\par
	\endgroup
	\fi}

\makeatother

\begin{document}
\hfuzz=4pt

\maketitle

\begin{abstract}
In 1953 LeVeque proved the existence of $U_m$-numbers by showing that for some specially defined  Liouville number~$\lambda$, the $m$th root $\lambda^{1/m}$ is in~$U_m$. In this article we study the following question: let~$u$ be an algebraic function of degree~$m$ and~$\lambda$ a Liouville number; under which conditions is $u(\lambda)$ a $U_m$-number? We consider a more refined notion of $\LL$-numbers, and show that, under very general assumptions, an algebraic function of degree~$m$ takes $U_m$-values at all $\LL$-numbers.

\end{abstract}

 {\footnotesize
 
\tableofcontents

}

\section{Introduction}
\label{sintro}

In this article~$m$ denotes a positive integer. LeVeque~\cite{Le53} introduced~$U_m$ as the set of the complex transcendental  numbers that can be well approximated by algebraic numbers of degree~$m$ but not by algebraic numbers of smaller degree. Precisely, a complex transcendental number~$\lambda$ belongs to~$U_m$ if~$m$ is the smallest integer with the following property:  there exists an infinite sequence of complex algebraic numbers $(\beta_n)_{n\ge0}$ of degree~$m$ and an infinite sequence of positive real numbers $(w_n)_{n\ge 0}$  such that 
\begin{equation}
\label{edefum}
\height(\beta_n)\to \infty, \qquad w_n\to \infty, \qquad |\lambda-\beta_n|\le e^{-wn\height(\beta_n)} \quad (n\ge 0). 
\end{equation}
Here ${\height(\cdot)}$ denotes the usual absolute logarithmic height. (LeVeque uses a different, but equivalent language,  see \cite[Remark~2.6]{BMM26}.) The elements of the set~$U_m$ are called \textit{$U_m$-numbers}. In particular, the Liouville numbers are exactly the $U_1$-numbers.

LeVeque \cite[Theorem~5]{Le53} shows that $U_m$-numbers exist for every~$m$, by showing that for some specially defined  Liouville number~$\lambda$, the $m$th root $\lambda^{1/m}$ is in~$U_m$; see also Bugeaud~\cite[Theorem~7.4]{Bu04}.

In this article we study the following question: given an algebraic function~$u$ of degree~$m$ and a Liouville number~$\lambda$, under which conditions can we  expect $u(\lambda)$ to be a~$U_m$-number?  

Here, by an \textit{algebraic function of degree~$m$ } we mean a holomorphic function~$u$, which is defined on a non-empty  open   subset\footnote{Since we are studying the values of our functions at Liouville numbers, which are real, we will always tacitly assume that the domain of definition of every algebraic function below has a non-empty intersection with~$\R$.} of~$\C$ and  is algebraic of degree~$m$ over the field of rational functions $\Q(z)$;   that is, there exists a polynomial ${P(X,Y) \in \Q[X,Y]}$, irreducible in the ring $\Q[X,Y]$, such that ${\deg_YP=m}$ and ${P(z,u(z))=0}$ for all~$z$ in the domain of definition of~$u$.  This polynomial will be called the \textit{minimal polynomial}  of~$u$; it is well-defined up to a non-zero constant factor. 

Of course, one cannot expect  $u(\lambda)$ to be a~$U_m$-number for any~$u$ of degree~$m$ and any Liouville number~$\lambda$. For instance, if ${\lambda=\kappa^2}$, where~$\kappa$ is another Liouville number, then $\lambda^{1/2}$ is not in~$U_2$. 

Moreover, assume that~$u$ is a \textit{real-valued} algebraic function of degree ${m>1}$; that is, there is a non-empty interval ${I\subseteq \R}$, where~$u$ is defined and takes real values. In this case, which is the most interesting in our context,~$u$ cannot take $U_m$-values at all Liouville numbers, by the following general principle, based on the Baire category theorem (see, for instance, \cite[Corollary~2]{STW14}).

\begin{proposition}
\label{prbaire}
Let ${I\subset \R}$ be a non-empty open interval and ${f:I\to\R}$ a continuous nowhere locally constant  function. Then there exist uncountably many Liouville numbers ${\lambda\in I}$ such that $f(\lambda)$ is a Liouville number as well. 
\end{proposition}


\subsection{$\LL$-numbers}
\label{ssll}

Thus, some conditions have to be imposed. Let~$\LL$ be the set of real numbers~$\lambda$ with the following property: there exist a sequence $(\alpha_n)_{n\ge 0}$ of rational numbers and 
a  sequence $(v_n)_{n\ge 0}$ of positive real numbers 
such that 
\begin{equation}
\label{eliouv}
\height(\alpha_n)\to\infty, \qquad v_n\to \infty, \qquad \left|\lambda-\alpha_n\right|\le e^{-v_n\height(\alpha_n)}  \qquad (n\ge 0), 
\end{equation}
and, moreover, there exists ${A>0}$  such that 
\begin{equation}
\label{ensparse}
\height(\alpha_{n+1})\le Av_n\height(\alpha_n) \qquad (n\ge 0). 
\end{equation}
Condition~\eqref{ensparse} 
can be informally expressed as ``good rational approximations for~$\lambda$ are not too sparse''.

The elements of the set~$\LL$ will be called $\LL$-numbers. Clearly, $\LL$-numbers are Liouville numbers, but the converse is not true: using continued fractions, it is easy to give an example of a Liouville number which is not in~$\LL$.

The $\LL$-numbers were introduced and explored in~\cite{CMT21} as a reasonable substitute for the notion of \textit{strong Liouville numbers}.  These latter were introduced by LeVeque~\cite{Le53} and studied by Erd\H os~\cite{Er62},  Alniaçik~\cite{Al79}, Petruska~\cite{Pe92} and others. A Liouville number~$\lambda$ is called \textit{strong}  if ${|\lambda-p_n/q_n|<q_n^{-w_n}}$ with ${w_n\to+\infty}$, where $(p_n/q_n)$ is the sequence of the convergents of the continued fraction of~$\lambda$.   LeVeque~\cite{Le53} claims that the Liouville number used in the proof of his Theorem~5 is   strong, without providing any justification. In fact, it is not true\footnote{Fortunately, this mistake does not affect the validity of LeVeque's  proof of  his Theorem~5.}, as follows from the work of Petruska~\cite{Pe92}. Actually, Petruska proves that neither of the ``famous'' Liouville numbers (including Liouville's original $\sum2^{-n!}$) is strong. Note that we are not aware of a single example of a strong Liouville number except those constructed using continued fractions with the special purpose to be strong.    

As follows from the standard properties of continued fractions, any strong Liouville number is an $\LL$-number. Moreover, $\sum2^{-n!}$ and similar numbers are $\LL$-numbers as well, which means that $\LL$-numbers is a much wider class than strong Liouville numbers.

Many known results about strong Liouville numbers extend, without much effort, to $\LL$-numbers. This concerns, in particular, Theorems~1 and~2 of Alniaçik~\cite{Al79}, see also \cite{CM14,CMT21}.

On the contrary, if one tries to relax the definition of $\LL$-numbers, many results no longer remain  true~\cite{CMT21}.   This implies that $\LL$-numbers is a reasonable subset of Liouville numbers to study: it is sufficiently rich to contain interesting examples, and sufficiently refined to have many desired properties. 

The reader may also consult the work of Kekeç \cite{Ke11,Ke11a,Ke26}; in particular, in the recent article~\cite{Ke26} the notions of $p$-adic $U_m$-number and $p$-adic strong Liouville number are explored, see Theorem~3.1 and Example~5.1 therein. 

\subsection{Results}

Our principal results are  three theorems below. Two of them concern values of algebraic functions at~$\LL$-numbers, and the third one at all the Liouville numbers. 

\begin{theorem}
\label{throot}
Assume that ${m>1}$ and denote by~$p$ the smallest prime divisor of~$m$. 
Let ${Q(z) \in \Q[z]}$ be a polynomial having at least~$k$ simple zeros in~$\bar\Q$, where 
\begin{equation}
\label{enumbersimpleroots}
k=
\begin{cases}
5, & \text{if $p=2$}, \\
4, & \text{if $p=3$}, \\
2, & \text{if $p\ge 5$}. 
\end{cases}
\end{equation}
Then for any~$\LL$-number~$\lambda$,  the $m$th root $Q(\lambda)^{1/m}$ is a $U_m$-number. (This is true for any definition of the $m$th root.) 
\end{theorem}

Theorem~\ref{throot} asserts that certain algebraic functions of special type (specifically, $m$th roots of polynomials)  take $U_m$-values at $\LL$-numbers. 
 The next theorem shows the same for general algebraic functions of degree~$m$ under some very mild conditions.

Let~$u$ be an algebraic function of degree~$m$. We call the \textit{genus} of~$u$ (notation: $\genus(u)$) the genus of the field $\Q(z,u)$. Let~$M$ be the Galois closure  of $\Q(z,u)$ over $\Q(z)$. The Galois group $\Gal(M/\Q(z))$, realized as a subgroup of the symmetric group $\SSS_m$, will be called  the \textit{Galois group} of~$u$  (notation: $\Gal(u)$). 

\begin{theorem}
\label{thgen}
Let~$u$ be an algebraic function of degree~$m$ such that ${\genus(u)\ge 2}$, and that ${\Gal(u)=\SSS_m}$ or~${\AAA_m}$. 
Then we have ${u(\lambda)\in U_m}$  for any $\LL$-number~$\lambda$ in the domain of definition of~$u$. 
\end{theorem}

The conditions on~$u$ imposed in this theorem are ``generic'', in the sense that a ``randomly selected'' algebraic function of degree~$m$ is expected to satisfy them. This vague statement can be made mathematically precise in many ways, but we prefer not to go into it here. 

\bigskip


As follows from Proposition~\ref{prbaire}, one needs to replace the set of Liouville numbers by a smaller subset (like $\LL$-numbers) if one works with a real-valued algebraic function. The situation changes if we consider algebraic functions taking non-real values on~$\R$. For instance, it is obvious that the algebraic function ${u(z):=z+i}$ of degree~$2$ takes $U_2$-values at all Liouville numbers, not just at $\LL$-numbers. More generally, upon imposing a mild restriction on a non-real algebraic number~$\gamma$ of degree~$m$, one can show that the function ${u(z):=z+\gamma}$ takes $U_m$-values at all Liouville numbers. 

Let~$\gamma$ be an algebraic number of degree~$m$ and let~$M$ be the Galois closure  of $\Q(\gamma)$ over $\Q$. The Galois group $\Gal(M/\Q)$, realized as a subgroup of the symmetric group $\SSS_m$, will be called  the \textit{Galois group} of~$\gamma$.

\begin{theorem}
\label{thcomplex}
Let~$\gamma$ be a non-real complex algebraic number of degree~$m$ and of Galois group~$\SSS_m$. Then for any Liuoville number~$\lambda$ we have ${\lambda+\gamma\in U_m}$. 
\end{theorem}

Note that it is not enough  to assume only that~$\gamma$ is non-real and of degree~$m$, one indeed has to impose some extra condition upon~$\gamma$. For instance, take ${\gamma:=\sqrt2+i}$, of degree~$4$. By Erd\H os~\cite{Er62} (or by Proposition~\ref{prbaire} above), there are uncountably many Liouville numbers~$\lambda$ such that ${\lambda+\sqrt2}$ is a Liouville number as well. For every such~$\lambda$, the number ${\lambda+\gamma}$ is in~$U_2$, not in~$U_4$. 

The proof of Theorem~\ref{thcomplex} extends (with obvious modifications)  to functions of the type ${u(z):= P(z)+\gamma}$, where~$\gamma$ is as in Theorem~\ref{thcomplex}, and ${P(X)\in \Q[X]}$ is a non-constant polynomial with rational coefficients. 
It would be interesting to discover more general classes of (non-real-valued) algebraic functions of degree~$m$ with $U_m$-values at all  Liouville numbers.

\paragraph{Acknowledgments}
We thank  Qing Liu, Tali Monderer, Peter Müller, Danny Neftin, Ricardo Pengo and Robert Wilms for their help in preparing this article.  

\section{Preliminaries}
\label{sprem}

Unless indicated otherwise, in this article:

\begin{itemize}
\item
capital letters $X,Y$ stand for independent variables;

\item
capital letters~$P$ and~$Q$ denotes polynomials;

\item
small letters $u,x,y$ stand for algebraic functions;

\item
small letters $v,w$ stand for positive real numbers (usually, as terms of sequences $(v_n)$ and $(w_n)$); 



\item
small letter~$z$ stands for the complex variable; 

\item
small Greek letters ${\alpha, \beta, \gamma, \xi}$ stand for algebraic numbers;

\item 
small Greek letters ${\kappa, \lambda}$ stand for transcendental numbers. 
\end{itemize}

\subsection{Heights}
We denote by $\height(\cdot)$ the usual absolute logarithmic height on the field of algebraic numbers~$\bar\Q$,   and call it simply \textit{height}. We will use (often without special reference) the standard properties of the height, like
$$
\height(\alpha+\beta)\le \height(\alpha)+\height(\beta)+\log 2, \quad \height(\alpha\beta) \le \height(\alpha)+\height(\beta) \qquad (\alpha,\beta \in \bar\Q)
$$
and so on, 
 as given, for instance, in \cite[Section~2]{BMM26}.  In particular, we will use Liouville's inequality as follows.

\begin{proposition}[Liouville's inequality]
\label{prli}
Let~$\beta$ and~$\beta'$ be distinct complex algebraic numbers of degrees~$d$ and~$d'$, respectively. Then 
$$
|\beta-\beta'| \ge e^{-dd'(\height(\beta)+\height(\beta')+\log2)}. 
$$
\end{proposition}

The proof can be found, for instance, in \cite[Theorem~{1.5.21}]{BG06}. 

We also need to compare the heights of two algebraic numbers satisfying an algebraic relation.

\begin{proposition}
\label{prcompare}
Let ${P(X,Y) \in \bar\Q[X,Y]}$ be a polynomial with algebraic coefficients, and let ${\alpha,\beta \in \bar\Q}$ be such that ${P(\alpha, \beta)=0}$. Assume that $P(X,Y)$ is not divisible by the polynomial ${X-\alpha}$ (that is, the polynomial $P(\alpha, Y)$ is not identically~$0$).  Then 
\begin{equation}
\label{ehco}
\height(\beta) \le N\height(\alpha) + O(1), 
\end{equation}
where ${N:=\deg_XP}$ and the implied constant depends only on the polynomial~$P$. 
\end{proposition}

The proof can be found, for instance, in \cite[Proposition~5.2]{BM06}\footnote{Attention: in~\cite{BM06} the multiplicative height ${\Height:= \exp\height}$ is used.}. 
Assuming the polynomial $P(X,Y)$ irreducible, much sharper results are known; see, for instance,~\cite{Ha17}. However,  simple estimate~\eqref{ehco} is sufficient for our purposes.

\section{Proof of Theorem~\ref{throot}}

Our principal tool will be the following \textit{gap principle}. 

\begin{theorem}
\label{thgap}
Let~$\kappa$ be a complex transcendental number and~$m$ a positive integer. 
Let $(\beta_n)_{n\ge0}$ be a sequence of algebraic numbers of degree not exceeding~$m$ and  $(w_n)_{n\ge0}$  a sequence  of positive real numbers. Assume that 
\begin{equation}
\label{eapprogabe}
\height(\beta_n) \to \infty, \qquad w_n \to \infty, \qquad |\kappa-\beta_n|\le Ce^{-w_n\height(\beta_n)} \quad (n\ge0)
\end{equation}
with some ${C>0}$. 
Assume further that there exists ${B>0}$ such that 
\begin{equation}
\label{egap}
\height(\beta_{n+1})\le Bw_n \height(\beta_n) \qquad (n \ge 0). 
\end{equation}  
Then there exist positive numbers~$c$ and~$\eta$ such that for every algebraic~$\gamma$  of degree not exceeding~$m$, distinct from  each of~$\beta_n$, we have
\begin{equation}
\label{elowerbebe}
|\kappa-\gamma|\ge ce^{-\eta\height(\gamma)}.
\end{equation}
\end{theorem}

Very informally, if~$\kappa$ has  a ``not too sparse'' set of very good approximations by algebraic numbers of bounded degree, then it cannot be well approximated by the algebraic numbers of bounded degree not belonging to this set.

In this form this theorem can be found as  Theorem~3.1 in~\cite{BMM26}, but results of this flavor appeared much earlier; see, for instance, Theorem~4 in LeVeque~\cite{Le53}. 

\begin{remark}
\label{reheinf}
In~\cite{BMM26}, instead of ${\height(\beta_n)\to \infty}$  it is assumed that ${\beta_n\to \kappa}$ (in the complex topology).  Actually, the two  conditions are equivalent in this context.   Indeed, if the sequence of algebraic numbers  $(\beta_n)$ converges to a transcendental number, then it may not have infinitely many equal terms; by Northcott's theorem, this implies that ${\height(\beta_n)\to \infty}$. Conversely, if ${\height(\beta_n)\to \infty}$ then the  two conditions ${w_n\to +\infty}$ and ${|\kappa-\beta_n|\le Ce^{-w_n\height(\beta_n)}}$ imply that ${\beta_n\to \kappa}$. 
\end{remark}

We are going to use the following immediate consequence.

\begin{corollary}
\label{cogap}
In the set-up of Theorem~\ref{thgap}, assume that all the numbers~$\beta_n$, with finitely many exceptions, are of degree exactly~$m$. Then~$\kappa$ is a $U_m$-number. 
\end{corollary}

We will also be using  the following well-known irreducibility criterion for  binomials; see,  for instance, Theorem~9.1  in~\cite[Chapter~VI]{La02}. 

\begin{proposition}
\label{prlang}
Let~$K$ be a field, ${\alpha\in K^\times}$ and~$m$  a positive integer. Assume 
that for all primes ${p\mid m}$ we have ${\alpha\notin K^p}$. If ${4\mid m}$ then we assume, in addition, that ${\alpha\notin -4K^4}$. Then the polynomial ${X^m - \alpha}$ is irreducible in $K[X]$. 
\end{proposition}

Finally, we are going to use the celebrated theorem of Faltings~\cite{Fa83}.

\begin{theorem}[Faltings]
\label{thfa}
Let ${P(X,Y)\in \Q[X,Y]}$ be a $\Q$-irreducible polynomial defining an algebraic curve ${P(X,Y)=0}$ of genus ${\genus\ge 2}$. Then the equation ${P(x,y)=0}$ has at most finitely many solutions  ${(x,y)\in \Q^2}$. 
\end{theorem}

(Here we deviate from the convention made in the beginning of Section~\ref{sprem}.)

\bigskip

For the proof of  Theorem~\ref{throot}, the following observation is crucial. 

\begin{proposition}
\label{prdegm}
In the set-up of Theorem~\ref{throot},  for all but finitely many ${\alpha\in \Q}$, the $m$th root $Q(\alpha)^{1/m}$ is of degree~$m$ over~$\Q$; in other words, the polynomial ${X^m-Q(\alpha)}$ is $\Q$-irreducible. 
\end{proposition}

\begin{proof}

Let~$q$ be a prime divisor of~$m$. Then ${q\ge p}$. (Recall that we denote by~$p$ the smallest prime divisor of~$m$.) Let~$\genus$ be the genus of the algebraic curve ${Y^q=Q(X)}$.  We want to show that ${\genus \ge 2}$. 

We apply the Riemann-Hurwitz formula to the branched cover defined by ${(X,Y)\mapsto X}$. It is of degree~$q$, and each simple zero of~$Q$ is a branch point with ramification index~$q$.  By the hypothesis, we have at least~$k$ such simple zeros, where~$k$ is defined in~\eqref{enumbersimpleroots}.  

We consider separately the cases ${q\nmid k}$ and ${q\mid k}$. 
When ${q\nmid k}$, we have one of the following: either ${q\mid \deg Q}$, in which case~$Q$ has one more zero of order not divisible by~$q$, that  also serves as a branch point with ramification index~$q$; or, ${q\nmid \deg Q}$, in which case~$\infty$ is a branch point with ramification index~$q$. Thus, we have at least ${k+1}$ branch points with ramification index~$q$. 
The Riemann-Hurwitz formula implies that
${2\genus -2 \ge (k+1)(q-1)-2q}$. It follows that
$$
\genus \ge \frac{(k-1)(q-1)}{2} \ge \frac{(k-1)(p-1)}{2}, 
$$
and we find  that ${\genus \ge 2}$  in all cases.  

If  ${q\mid k}$ then  ${p=2}$ and ${q=k=5}$, in which case we have at least~$5$ branch points with ramification index~$5$, and the  Riemann-Hurwitz formula implies that ${\genus \ge 6}$.  

We have just proved that  the curve ${Y^q=Q(X)}$ is of genus at least~$2$ for every ${q\mid m}$. 

Also, if ${4\mid m}$, then ${p=2}$ and ${k=5}$. A similar application of the Riemann-Hurwitz formula implies that  the curve ${-4Y^4=Q(X)}$ is of genus at least~$2$ (in fact, even at least~$5$). 

By Theorem~\ref{thfa}, for every prime ${q\mid m}$ there may exist at most finitely many ${\alpha\in \Q}$ such that ${Q(\alpha) \in \Q^q}$. Also, when ${4\mid m}$, there may exists at most finitely many ${\alpha\in \Q}$ such that ${Q(\alpha) \in -4\Q^4}$. Using Proposition~\ref{prlang}, we complete the proof. 
\end{proof}

Now we are ready to prove Theorem~\ref{throot}. Let~$\lambda$ be an $\LL$-number, and let ${\kappa:= \Q(\lambda)^{1/m}}$ be some definition of the $m$th root of $\Q(\lambda)$.  Set 
$$
\eps:=\frac12\min\{|\lambda-\alpha|: Q(\alpha)=0\}, 
$$ 
that is, half of the smallest distance between~$\lambda$ and a zero of~$Q$. Since~$\lambda$ is transcendental, we have ${Q(\lambda) \ne 0}$, which implies that ${\eps>0}$.  Denote 
$$
\OO_\eps:= \{z\in \C: |z-\lambda|<\eps\}
$$
the $\eps$-neighborhood of~$\lambda$, and let $u(z)$ be the algebraic function on $\OO_\eps$ satisfying 
$$
u(z)^m=Q(z), \qquad u(\lambda)=\kappa. 
$$
Our definition of~$\eps$ implies that the derivative $u'(z)$ is bounded: there exists ${C>0}$, depending on~$Q$ and~$\lambda$, such that 
\begin{equation}
\label{ederiv}
|u'(z)|\le C \qquad (z\in \OO_\eps). 
\end{equation}
Proposition~\ref{prdegm} implies that for all ${\alpha \in \Q\cap\OO_\eps}$ with finitely many exceptions, $u(\alpha)$ is an algebraic number of degree~$m$. 

Let $(\alpha_n)$, $(v_n)$ and~$A$ be as in the beginning of Subsection~\ref{ssll}, so that conditions~\eqref{eliouv} and~\eqref{ensparse} are satisfied. By discarding finitely many~$n$, we may assume that ${\alpha_n\in  \OO_\eps}$ for all~$n$, and that for all~$n$ the algebraic number ${\beta_n:=u(\alpha_n)}$ is of degree~$m$. 

Applying  Proposition~\ref{prcompare} to the equation ${\beta_n^m=Q(\alpha_n)}$, we obtain, for each~$n$, the inequalities 
$$
\height(\beta_n) \le N\height(\alpha_n) +O(1), \qquad \height(\alpha_n) \le m\height(\beta_n)+O(1), 
$$
where ${N:=\deg Q(X)}$ and the implied constants depend on~$Q$ and~$m$. In particular, ${\height(\beta_n) \to \infty}$.
By discarding finitely many~$n$, we may assume that 
\begin{equation}
\label{ehcan}
\height(\beta_n) \le 2N\height(\alpha_n) , \qquad \height(\alpha_n) \le 2m\height(\beta_n)
\end{equation}
for all~$n$.  Using~\eqref{ederiv}, we obtain
$$
|\kappa-\beta_n|=|u(\lambda) - u(\alpha_n)|\le C|\lambda-\alpha_n| \le Ce^{-v_n\height(\alpha_n)} \le  Ce^{-(v_n/2N)\height(\beta_n)}. 
$$
This proves~\eqref{eapprogabe} with ${w_n:=v_n/2N}$. Furthermore, using~\eqref{ensparse}, we obtain 
$$
\height(\beta_{n+1}) \le 2N\height(\alpha_{n+1}) \le 2NAv_n \height(\alpha_n) \le 8N^2m Aw_n\height(\beta_n), 
$$
which proves~\eqref{egap} with ${B=:8N^2mA}$. Thus, Theorem~\ref{thgap} applies, and so does Corollary~\ref{cogap}, because all the~$\beta_n$ are of degree~$m$. Theorem~\ref{throot} is proved.   \qed

\section{Hilbert-Müller irreducibility theorem over~$\Q$}
\label{shimu}

This section is independent of the other parts of the article. The contents of this section emerged from discussions with Tali Monderer, Peter Müller and Danny Neftin. We thank them most warmly for their generous help. 

\bigskip

Let ${P(X,Y)\in \Q[X,Y]}$ be a $\Q$-irreducible polynomial.  Classical \textit{Hilbert's irreducibility theorem} asserts that, for infinitely many ${\xi \in \Z}$, the specialized polynomial ${P(\xi,Y)\in \Q[Y]}$ is $\Q$-irreducible. However, there can still  be infinitely many ${\xi\in \Z}$ for which $P(\xi, Y)$ is $\Q$-reducible, as the example ${P(X,Y):=Y^2-X}$ shows.

Müller~\cite{Mu99} showed that, imposing on~$P$ extra conditions, one would obtain a much stronger statement: $P(\xi,Y)$ is $\Q$-irreducible for all ${\xi \in \Z}$ with finitely many exceptions.  

Let us introduce some notation. We set 
${m:=\deg_YP}$, and we denote by ${G=\Gal_Y(P)}$ the Galois group of $P(x,Y)$ over the field $\Q(x)$ of rational functions, realized as a subgroup of the symmetric group~$\SSS_m$.  Finally, we denote by~$\genus$ the genus of the algebraic curve  ${P(X,Y)=0}$. 

The following is Theorem~1.2 from~\cite{Mu99}. 

\begin{theorem}[Müller]
Assume that ${G=\SSS_m}$ or that~$m$ is a prime number. Assume also that ${\genus\ge1}$. Then $P(\xi,Y)$ is $\Q$-irreducible for all but finitely many ${\xi \in \Z}$. 
\end{theorem}

We need an analogous result extending to ${\xi \in \Q}$. In this section we prove the following. 

\begin{theorem}
\label{thhimuq}
Assume that ${G=\SSS_m}$, or that ${G=\AAA_m}$ and ${m\ge 3}$.   Assume also that ${\genus\ge2}$. Then $P(\xi,Y)$ is $\Q$-irreducible for all but finitely many ${\xi \in \Q}$. 
\end{theorem}

Müller's argument uses Siegel's theorem about finitude of integral points on a curve of positive genus, together with some non-trivial group theory. In the argument below we use, as one may expect, Faltings' theorem (see Theorem~\ref{thfa}), and very little group theory; the only group-theoretic result involved is inequality~\eqref{egs} below, due to Guralnick  and Shareshian~\cite{GS07}.

\bigskip

We start from some preparations. Fix an algebraic closure of the field $\Q(x)$, and let ${u_1, \ldots, u_m}$ be the roots of $P(x,Y)$ (viewed as a polynomial in~$Y$ over $\Q(x)$)   in this algebraic closure.  Then $\Q(x,u_1, \ldots, u_m)$ is the splitting field of~$P$ over $\Q(x)$, and 
${G= \Gal \bigl(\Q(x,u_1, \ldots, u_m)/\Q(x)\bigr)}$. 

The group~$G$ acts by permutations on the set ${\{u_1, \ldots, u_m\}}$. Let~$A$ be a subset of ${\{u_1, \ldots, u_m\}}$, and~$G_A$ the subgroup of~$G$ stabilizing~$A$ as a set. The field 
${L_A:= \Q(x,u_1, \ldots, u_m)^{G_A}}$
is generated over $\Q(x)$ by the elementary symmetric functions of the elements of~$A$. For instance, if ${A=\{u_1,u_3,u_4\}}$, then 
$$
L_A=\Q(x,u_1+u_3+u_4,u_1u_3+u_1u_4+u_3u_4,u_1u_3u_4). 
$$
For every such~$A$ we fix a generator~$y_A$ of~$L_A$ over~$\Q(x)$, so that ${L_A=\Q(x,y_A)}$. Let ${P_A(X,Y)\in \Q[X,Y]}$ be the $\Q$-irreducible polynomial such that  
$$
P_A(x,y_A)=0; 
$$
it is well-defined up to a constant factor.


The standard reduction argument (see, for instance, Chapter~9 in~\cite{La83}) implies the following.

\begin{proposition}
\label{plangred}
Let ${1\le k\le m-1}$. Then for all but finitely many ${\xi \in \Q}$  the following conditions are equivalent. 

\begin{enumerate}
\item
The polynomial ${P(\xi, Y)}$ has a factor in $\Q[Y]$ of degree~$k$. 

\item
\label{ihasaroot}
For some $k$-element set ${A\subset\{u_1, \ldots, u_m\}}$, the polynomial ${P_A(\xi, Y)}$ has a root in~$\Q$. 

\end{enumerate}
\end{proposition}
If~$\tily_A$ is a different generator of~$L_A$ and $\tilP_A(X,Y)$ the corresponding polynomial, then for all but finitely many ${\xi \in \Q}$, the polynomial
 ${P_A(\xi, Y)}$ has a root in~$\Q$ if and only if so does $\tilP_A(X,Y)$. 
This  implies that we are free to choose~$y_A$ in the most convenient way. 

If $A,B$ are two subsets as above and ${B=\sigma(A)}$ for some ${\sigma \in G}$, then we may take ${y_B=\sigma(y_A)}$, and the polynomials~$P_A$ and~$P_B$ are the same. Hence, selecting a representative in every orbit of the $G$-action on the  subsets of ${\{u_1, \ldots, u_m\}}$, the set~$A$ in property~\ref{ihasaroot} of Proposition~\ref{plangred} may be chosen among the selected representatives. 

When ${G=\SSS_m}$ or ${G=\AAA_m}$ and ${m\ge 4}$, the action is $k$-transitive for every~$k$. Hence we may take as the representatives the sets 
${A_k:=\{u_1, \ldots, u_k\}}$.  
To simplify notation, we write $L_k$ and $P_k(X,Y)$ instead of $L_{A_k}$ and $P_{A_k}(X,Y)$. Selecting ${y_A=u_1}$ for ${A=A_1=\{u_1\}}$, we may assume that ${P_1=P}$.

We obtain the following consequence of Proposition~\ref{plangred}. 

\begin{proposition}
\label{pred}
Assume that ${G=\SSS_m}$ or that ${G=\AAA_m}$ and ${m\ge 4}$. Then for all but finitely many ${\xi \in \Q}$, 
\begin{equation*}
\text{${P(\xi, Y)}$ has a $\Q$-factor of degree~$k$}\ \Leftrightarrow\ \text{${P_k(\xi, Y)}$ has a root in~$\Q$}. 
\end{equation*}
In particular  (and also in the case ${m=3}$ and ${G=\AAA_3}$), we have
\begin{equation}
\label{eqred}
\text{${P(\xi, Y)}$ is $\Q$-reducible}\ \Leftrightarrow\ \text{${P_k(\xi, Y)}$ has a root in~$\Q$ for some $k\le m/2$}. 
\end{equation} 
\end{proposition}
Note that in the case ${m=3}$ and ${G=\AAA_3}$ equivalence~\eqref{eqred} holds trivially, because ${P_1=P}$.

\bigskip

Denote by $\genus_k$ the genus of the field~$L_k$; it is also the genus of the curve ${P_k(X,Y)=0}$. In particular, ${\genus_1=\genus}$, the genus of ${P(X,Y)=0}$.  According to \cite[Lemma~2.0.12]{GS07}\footnote{In~\cite{GS07} the authors use the language of Riemann surfaces and work over~$\C$, so their result applies to the subfields of  $\C(x,u_1, \ldots, u_m)$ rather than $\Q(x,u_1, \ldots, u_m)$. However, since the Galois group is~$\SSS_m$ or~$\AAA_m$, the constant subfield of the field $L_{k}$  is~$\Q$ for every~$k$, and changing the base field from~$\Q$ to~$\C$ does not affect the genus.}, when ${G=\SSS_m}$ or~$\AAA_m$, we have the inequality
\begin{equation}
\label{egs}
\genus_1 \le \genus_2\le \cdots \le \genus_{\lfloor m/2\rfloor}. 
\end{equation}
Now we are ready to prove Theorem~\ref{thhimuq}.

\begin{proof}[Proof of Theorem~\ref{thhimuq}]
Since ${\genus_1=\genus\ge2}$ by the hypothesis, we have ${\genus_k\ge 2}$ for ${1\le k\le m/2}$. By Faltings' theorem, for every such~$k$ there may exist at most finitely many ${\xi\in \Q}$ such that $P_k(\xi,Y)$ has a root in~$\Q$. Now the result follows by Proposition~\ref{pred}. 
\end{proof}

\begin{remark}
Monderer and Neftin \cite[Theorem~1.1]{MN22} obtain, for~$m$ exceeding some absolute constant\footnote{As indicated in~\cite[Remark~2.8]{MN22}, a suitable value for the ``absolute constant'' is $6\cdot 10^{10}$ (or ${3.5\cdot 10^6}$ if one is ready to accept the classification of the finite simple groups).}, the following much stronger result: if ${G=\SSS_m}$ or~$\AAA_m$ (and without any hypothesis about the genus), for all but finitely many ${\xi\in \Q}$ the Galois group of the polynomial $P(\xi,Y)$ is one of
$$
\SSS_m,\ \AAA_m,\ \SSS_{m-1},\ \AAA_{m-1},\ \SSS_{m-2},\ \AAA_{m-2}\times \SSS_2,\ \SSS_{m-2}\times \SSS_2. 
$$
In the set-up of Theorem~\ref{thhimuq}, this implies that,  for all but finitely many ${\xi\in \Q}$, the Galois group of $P(\xi,Y)$ is~$\SSS_m$ or~$\AAA_m$: indeed, in this case ${\genus_2\ge\genus_1\ge 2}$ by~\eqref{egs}, and the other groups from the list above may occur only for finitely many~$\xi$ by Faltings' theorem and Proposition~\ref{pred}. The only advantage of our Theorem~\ref{thhimuq} compared to the result of Monderer and Neftin is that we do not impose any restriction on~$m$. 
\end{remark}

\section{Proof of Theorem~\ref{thgen}}

The proof goes along the same lines as that of Theorem~\ref{throot}, but instead of Proposition~\ref{prdegm} we use Theorem~\ref{thhimuq}. 

Let~$u$ be an algebraic function satisfying the hypothesis  of Theorem~\ref{thgen}. Then there exists a $\Q$-irreducible polynomial ${P(X,Y) \in \Q[X,Y]}$, satisfying the hypothesis of Theorem~\ref{thhimuq} and such that ${P(z,u(z))=0}$ for all~$z$ in the domain of definition of~$u$. In particular, ${\deg_YP=m}$. 

(The statement of Theorem~\ref{thgen} does not explicitly mention that ${m\ge 3}$ if ${G=\AAA_m}$, as required in the hypothesis of Theorem~\ref{thhimuq}. However, one cannot have ${m=1}$ because ${\genus >0}$, and one cannot have ${m=2}$ because~$\AAA_2$ cannot be the Galois group of an algebraic function of degree~$2$.)

Let~$\lambda$ be an $\LL$-number in the domain of definition of~$u$. We want to prove that ${\kappa:=u(\lambda)}$ is a $U_m$-number. Let ${\eps>0}$ be such that~$u$ is defined in the $\eps$-neighborhood
${\OO_\eps:= \{z\in \C: |z-\lambda|<\eps\}}$ 
and has  bounded derivative therein. 
Theorem~\ref{thhimuq} implies that for all but finitely many ${\alpha \in \Q\cap\OO_\eps}$, the algebraic number $u(\alpha)$ is of degree~$m$.

Let $(\alpha_n)$ and $(v_n)$ be as in the proof of Theorem~\ref{throot}. By discarding finitely many~$n$, we may assume that every~$\alpha_n$ belongs to~$\OO_\eps$, every ${\beta_n:=u(\alpha_n)}$ is of degree~$m$, and inequalities~\eqref{ehcan} hold for all~$n$, where ${N:=\deg_XP}$. Now we complete the proof exactly as we did for Theorem~\ref{throot}. \qed

\section{Proof of Theorem~\ref{thcomplex}}

The proof relies on a nice lemma that we learned from Robert Wilms (personal communication).

\begin{lemma}[Wilms]
\label{lwilms} 
Let~$\gamma$ be a non-real complex algebraic number of degree~$m$ and Galois group~$\SSS_m$. Then ${\im \gamma \ne \im \beta}$ for any complex algebraic~$\beta$ of degree strictly smaller than~$m$.   
\end{lemma}

\begin{proof}
We will prove a more general statement: let~$\gamma$ be algebraic of degree~$m$ and Galois group~$\SSS_m$, and ${\gamma'\ne \gamma}$  a conjugate of~$\gamma$ over~$\Q$; then ${\gamma-\gamma'\ne \beta-\beta'}$ for any algebraic $\beta,\beta'$ of degree strictly smaller than~$m$.  Lemma~\ref{lwilms} is a special case, with ${\gamma':=\bar{\gamma}}$  and ${\beta':=\bar{\beta}}$, where ${z\mapsto\bar{z}}$ denotes complex conjugation. 

It suffices to prove that ${\gamma-\gamma'}$ is of degree ${m(m-1)}$, because the degree of ${\beta-\beta'}$ cannot exceed ${(m-1)^2}$. Let ${\gamma_1, \ldots, \gamma_m}$ be the conjugates of~$\gamma$ over~$\Q$.  Since the Galois group~$\SSS_m$ is doubly transitive, the ${m(m-1)}$ differences 
\begin{equation}
\label{ediffs}
\gamma_j-\gamma_k \qquad (1\le j,k\le m, \quad j\ne k)
\end{equation}
are all conjugate over~$\Q$, and we only have to show that they are distinct.

Thus, assume that ${\gamma_1-\gamma_2=\gamma_j-\gamma_k}$ for some ${(j,k)\ne (1,2)}$. If ${j=1}$ then ${\gamma_k=\gamma_2}$, a contradiction, and same for ${k=2}$. If ${(j,k)=(2,1)}$ then ${\gamma_1=\gamma_2}$, again a contradiction. 

If ${j=2}$ but ${k\ne 1}$, then we may assume that ${k=3}$, and we obtain 
\begin{equation}
\label{em=three}
\gamma_3=2\gamma_2-\gamma_1\in \Q(\gamma_1,\gamma_2). 
\end{equation}
Since the Galois group is~$\SSS_m$, this is possible only if ${m=3}$, in which case ${\gamma_1+\gamma_2+\gamma_3\in \Q}$. Together with~\eqref{em=three} this implies that ${\gamma_2\in \Q}$, a contradiction. Same for ${k=1}$ but ${j\ne 2}$.

Thus, we have ${j,k\ne 1,2}$, and we may assume that ${\gamma_1-\gamma_2=\gamma_3-\gamma_4}$, which implies that ${m=4}$. Hence  ${\gamma_1+\cdots+\gamma_4\in \Q}$, and we obtain ${\gamma_1+\gamma_4\in \Q}$, a contradiction. 
\end{proof}

Now we are ready to prove Theorem~\ref{thcomplex}.  Since~$\lambda$ is a Liouville number, there exist sequences of rational numbers $(\alpha_n)_{n\ge0}$ and of positive real numbers $(v_n)_{n\ge0}$ such that 
$$
\height(\alpha_n)\to \infty, \qquad v_n\to \infty, \qquad |\lambda-\alpha_n|\le e^{-u_n\height(\alpha_n)} \quad (n\ge 0). 
$$
The algebraic numbers ${\beta_n:= \gamma+\alpha_n}$ are all of degree~$m$, and satisfy 
$$
 \height(\alpha_n)-\height(\gamma)-\log2\le \height(\beta_n)\le \height(\alpha_n)+\height(\gamma)+\log2. 
$$
By discarding finitely many~$n$, we may assume that ${\height(\alpha_n)/2\le \height(\beta_n)\le 2\height(\alpha_n)}$ for all~$n$. Hence, setting ${\kappa:=\lambda+\gamma}$ and ${w_n:=v_n/2}$, we obtain a sequence $(\beta_n)$ of algebraic numbers of degree~$m$ and a sequence $(w_n)$ of positive real numbers such that 
$$
\height(\beta_n)\to \infty, \qquad w_n\to \infty, \qquad |\kappa-\beta_n|\le e^{-w_n\height(\beta_n)} \quad (n\ge 0).   
$$
To complete the proof, we have to show that there exists ${C>0}$ such that 
\begin{equation}
\label{egebeta}
|\kappa-\beta|\ge e^{-C(\height(\beta)+1)}
\end{equation}
for any algebraic~$\beta$ of degree strictly smaller than~$m$. 

Thus, let~$\beta$ be as above. Since ${\lambda \in \R}$, we have 
${|\kappa-\beta|\ge  |\im\gamma-\im\beta|}$. Lemma~\ref{lwilms} implies that ${\im\gamma\ne \im\beta}$. Note that  $\im\beta$ is of degree at most $m^2$; also,  from ${\im\beta=(\beta-\bar\beta)/2}$ we deduce
${\height(\im\beta) \le 2\height(\beta)+ 2\log 2}$. 
Liouville's inequality as in Proposition~\ref{prli}, applied to $\im\gamma$ and $\im\beta$,  implies that~\eqref{egebeta} holds 
with some ${C>0}$ depending only on~$\gamma$. The theorem is proved.

{\footnotesize

\bibliographystyle{amsplain}
\bibliography{alg_fun_liou}

\providecommand{\bysame}{\leavevmode\hbox to3em{\hrulefill}\thinspace}
\providecommand{\MR}{\relax\ifhmode\unskip\space\fi MR }
\providecommand{\MRhref}[2]{%
  \href{http://www.ams.org/mathscinet-getitem?mr=#1}{#2}
}
\providecommand{\href}[2]{#2}
\begin{thebibliography}{10}

\bibitem{Al79}
K{â}mil Alnia{ç}ik, \emph{On the subclasses {$U\sb{m}$}\ in {M}ahler's
  classification of the transcendental numbers}, \.Istanbul \"Univ. Fen Fak.
  Mecm. Ser. A \textbf{44} (1979), 39--82. \MR{728149}

\bibitem{BMM26}
Yuri Bilu, Diego Marques, and Carlos~Gustavo Moreira, \emph{{Values of
  generalized Liouville power series at algebraic numbers}}, Funct. Approx.
  Comment. Math. (2026).

\bibitem{BM06}
Yuri Bilu and David Masser, \emph{A quick proof of {S}prindzhuk's decomposition
  theorem}, More sets, graphs and numbers, Bolyai Soc. Math. Stud., vol.~15,
  Springer, Berlin, 2006, pp.~25--32. \MR{2223386}

\bibitem{BG06}
Enrico Bombieri and Walter Gubler, \emph{Heights in {D}iophantine geometry},
  New Mathematical Monographs, vol.~4, Cambridge University Press, Cambridge,
  2006. \MR{2216774}

\bibitem{Bu04}
Yann Bugeaud, \emph{Approximation by algebraic numbers}, Cambridge Tracts in
  Mathematics, vol. 160, Cambridge University Press, Cambridge, 2004.
  \MR{2136100}

\bibitem{CM14}
Ana~Paula Chaves and Diego Marques, \emph{An explicit family of
  {$U_m$}-numbers}, Elem. Math. \textbf{69} (2014), no.~1, 18--22. \MR{3182261}

\bibitem{CMT21}
Ana~Paula Chaves, Diego Marques, and Pavel Trojovsk\'y, \emph{On the arithmetic
  behavior of {L}iouville numbers under rational maps}, Bull. Braz. Math. Soc.
  (N.S.) \textbf{52} (2021), no.~4, 803--813. \MR{4325883}

\bibitem{Er62}
P.~Erd{\H{o}}s, \emph{Representations of real numbers as sums and products of
  {L}iouville numbers}, Michigan Math. J. \textbf{9} (1962), 59--60.
  \MR{133300}

\bibitem{Fa83}
G.~Faltings, \emph{Endlichkeitss\"atze f\"ur abelsche {V}ariet\"aten \"uber
  {Z}ahlk\"orpern}, Invent. Math. \textbf{73} (1983), no.~3, 349--366.
  \MR{718935}

\bibitem{GS07}
Robert~M. Guralnick and John Shareshian, \emph{Symmetric and alternating groups
  as monodromy groups of {R}iemann surfaces. {I}. {G}eneric covers and covers
  with many branch points}, Mem. Amer. Math. Soc. \textbf{189} (2007), no.~886,
  vi+128, With an appendix by Guralnick and R. Stafford. \MR{2343794}

\bibitem{Ha17}
Philipp Habegger, \emph{Quasi-equivalence of heights and {R}unge's theorem},
  Number theory---{D}iophantine {P}roblems, {U}niform {D}istribution and
  {A}pplications, Springer, Cham, 2017, pp.~257--280. \MR{3676405}

\bibitem{Ke11}
G{ü}lcan Keke{ç}, \emph{On some lacunary power series with algebraic
  coefficients and {M}ahler's {$U$}-numbers}, Appl. Math. Comput. \textbf{218}
  (2011), no.~3, 866--870. \MR{2831323}

\bibitem{Ke11a}
\bysame, \emph{On the values of some generalized lacunary power series with
  algebraic coefficients for {L}iouville number arguments}, Hacet. J. Math.
  Stat. \textbf{40} (2011), no.~5, 691--702. \MR{2896066}

\bibitem{Ke26}
\bysame, \emph{An explicit construction of $p$-adic $u_{m}$-numbers}, Indian J.
  Pure Appl. Math. (2026).

\bibitem{La83}
Serge Lang, \emph{Fundamentals of {D}iophantine geometry}, Springer-Verlag, New
  York, 1983. \MR{715605}

\bibitem{La02}
\bysame, \emph{Algebra}, third ed., Graduate Texts in Mathematics, vol. 211,
  Springer-Verlag, New York, 2002. \MR{1878556}

\bibitem{Le53}
W.~J. LeVeque, \emph{On {M}ahler's {$U$}-numbers}, J. London Math. Soc.
  \textbf{28} (1953), 220--229. \MR{54658}

\bibitem{MN22}
Tali Monderer and Danny Neftin, \emph{Symmetric {G}alois groups under
  specialization}, Israel J. Math. \textbf{248} (2022), no.~1, 201--227.
  \MR{4429280}

\bibitem{Mu99}
Peter M\"uller, \emph{Hilbert's irreducibility theorem for prime degree and
  general polynomials}, Israel J. Math. \textbf{109} (1999), 319--337.
  \MR{1679603}

\bibitem{Pe92}
G.~Petruska, \emph{On strong {L}iouville numbers}, Indag. Math. (N.S.)
  \textbf{3} (1992), no.~2, 211--218. \MR{1168349}

\bibitem{STW14}
K.~Senthil~Kumar, R.~Thangadurai, and M.~Waldschmidt, \emph{Liouville numbers
  and {S}chanuel's {C}onjecture}, Arch. Math. (Basel) \textbf{102} (2014),
  no.~1, 59--70. \MR{3154158}

\end{thebibliography}

\bigskip

\noindent 
Yuri Bilu: 
Institut de Mathématiques de Bordeaux, Université de Bordeaux \& CNRS, Talence, France; 
\url{yuri@math.u-bordeaux.fr}

\bigskip

\noindent
Diego Marques: Departamento de Matemática, Universidade de Brasília, Brasília, DF, Brazil; 
\url{diego@mat.unb.br}

}

\end{document}